\documentclass[runningheads]{llncs}
\usepackage{graphicx}
\usepackage{tikz}
\usetikzlibrary{graphs}
\usetikzlibrary{matrix}
\usetikzlibrary{backgrounds}
\usetikzlibrary{fit}
\usepackage{amsmath, amsfonts, amssymb}
\usepackage[linesnumbered,lined,ruled,commentsnumbered]{algorithm2e}
\newtheorem{observation}[theorem]{Observation}
\begin{document}
\title{Equitable 
$d$-degenerate
choosability of graphs\thanks{The short version of the paper was accepted to be published in proceedings of IWOCA2020}}
%
%
\author{Ewa Drgas-Burchardt\inst{1} \and
Hanna Furma\'nczyk\inst{2} \and
El\.zbieta Sidorowicz\inst{1}}
\authorrunning{E. Drgas-Burchardt, H. Furma\'nczyk, E. Sidorowicz}
%
\institute{Faculty of Mathematics, Computer Science and Econometrics,\ University of Zielona G\'ora,\ Prof. Z. Szafrana 4a,\ 65-516 Zielona G\'ora,\ Poland
\email{\{E.Drgas-Burchardt,E.Sidorowicz\}@wmie.uz.zgora.pl}\\
\and
Institute of Informatics,\ Faculty of Mathematics, Physics and Informatics,\ University of Gda\'nsk,\ 80-309 Gda\'nsk,\ Poland\\
\email{hanna.furmanczyk@ug.edu.pl}}
\maketitle             
\begin{abstract}
Let ${\mathcal D}_d$ be the class of $d$-degenerate graphs and let $L$ be a list assignment for a graph $G$. 
A colouring of $G$ such that every vertex receives a colour from its list and the subgraph induced by vertices coloured with one color is a $d$-degenerate graph is called the $(L,{\mathcal D}_d)$-colouring of $G$.
For a $k$-uniform list assignment $L$ and $d\in\mathbb{N}_0$, a graph $G$ is equitably $(L,{\mathcal D}_d)$-colorable if  there is an $(L,{\mathcal D}_d)$-colouring of $G$ such that the size of any colour class does not exceed $\left\lceil|V(G)|/k\right\rceil$.
An equitable $(L,{\mathcal D}_d)$-colouring is a generalization of an equitable list coloring, introduced by Kostochka at al., and an equitable list arboricity presented by Zhang. Such a model can be useful in the network decomposition where some structural properties on subnets are imposed.

In this paper we give a polynomial-time algorithm that for a given $(k,d)$-partition of $G$ with a $t$-uniform list assignment $L$ and $t\geq k$, returns its equitable $(L,\mathcal{D}_{d-1})$-colouring. In addition, we show that  3-dimensional grids are equitably  $(L,\mathcal{D}_1)$-colorable for any $t$-uniform list assignment $L$  where $t\geq 3$.

\end{abstract}
\section{Motivation and preliminaries}
In last decades, a social network graphs, describing relationship in real life, started to be very popular and present everywhere. Understanding key structural properties of large-scale data networks started to be crucial for analyzing and optimizing their performance, as well as improving their security. This topic has been attracting attention of many researches, recently (see \cite{dragan,lei,miao,deg}). We consider one of problems connected with the decomposition of networks into smaller pieces fulfilling some structural properties. For example, we may desire that, for some security reason, the pieces are acyclic or even independent. This is because  of in such a piece we can easily and effectively identify a node failure since the local structure around such a node in this piece is so clear that it can be easily tested using some classic algorithmic tools \cite{deg}. Sometimes, it is also desirable that the sizes of pieces are balanced. It helps us to maintain the whole communication network effectively. Such a problem can be modeled by minimization problems in graph theory, called an \emph{equitable vertex arboricity} or an  \emph{equitable vertex colourability} of graphs. Sometimes we have some additional requirements on vertices/nodes that can be modeled by a list of available colours. So, we are interested in the list version, introduced by Kostochka, Pelsmajer and West \cite{Ko03} (an independent case), and by Zhang \cite{Zh16} (an acyclic case).

In colourability and arboricity models the properties of a network can be described in the language of the upper bound on the minimum degree, i.e. each colour class induces a graph whose each induced subgraph has the minimum degree bounded from above by zero or one, respectively. In the paper we consider the generalization of these models in which each colour class induces a graph whose each induced subgraph has the minimum degree bounded from above by some natural constant.

Let   $\mathbb{N}_0 = \mathbb{N}\cup\{0\}$. For  $d\in \mathbb{N}_0$, the graph $G$ is $d$-{\it degenerate} if $\delta(H)\le d$ for any subgraph $H$ of $G$, where $\delta(H)$ denotes the minimum degree of $H$. The class of all $d$-degenerate graphs is denoted by ${\cal D}_d$. In particular, ${\cal D}_0$ is the class of all edgeless graphs and ${\cal D}_1$ is the class of all forests. A ${\cal D}_d$-{\it coloring} is a mapping $c:V(G)\rightarrow \mathbb{N}$ such that for each $i\in \mathbb{N}$ the set of vertices coloured with $i$ induces a $d$-degenerate graph.
A \textit{list assignment} $L$, for a graph $G$, is a mapping that assigns a nonempty subset of  $\mathbb{N}$ to each vertex $v\in V(G)$. Given $k\in \mathbb{N}$, a list assignment $L$ is $k$-\textit{uniform} if $|L(v)|=k$ for every $v\in V(G)$. 
Given $d\in \mathbb{N}_0$, a graph $G$ is  $(L,{\mathcal D}_d)$-{\it colourable} if there exists a colouring $c: V(G) \to \mathbb{N}$, such that $c(v) \in L(v)$ for each $v \in V(G)$, and  for each $ i\in \mathbb{N}$ the set of vertices coloured with $i$ induces a $d$-degenerate graph. Such a mapping $c$ is called  an {\it $(L,{\mathcal D}_d)$-colouring} of $G$. The $(L,{\mathcal D}_d)$-colouring of $G$ is also named as its $L$-colouring. 
If $c$ is any colouring of $G$, then its restriction to $V^\prime$, $V^\prime \subseteq V(G)$, is denoted by $c|_{V'}$. For a partially coloured graph $G$, let 
$N_G^{col}(d,v)=\{w\in N_G(v): w\; \mbox{has}\; d \;\mbox{neighbors coloured with}\; c(v)\}$, where $N_G(v)$ denotes the set of vertices of $G$  adjacent to $v$. We refer the reader to \cite{Dist00} for
terminology not defined in this paper.

Given $k\in \mathbb{N}$ and $d\in \mathbb{N}_0$, a graph $G$ is \textit{equitably $(k,{\mathcal D}_d)$-choosable} if for any $k$-uniform list assignment $L$ there is an $(L,{\mathcal D}_0)$-colouring of $G$ such that the size of any colour class does not exceed
$\left\lceil |V(G)|/k\right\rceil$.
The notion of equitable $(k,{\mathcal D}_0)$-choosability was introduced by Kostochka et al. \cite{Ko03} whereas the notation of equitable $(k,{\mathcal D}_1)$-choosability was introduced by Zhang \cite{Zh16}.

Let $k,d\in\mathbb{N}$.  A partition $S_1\cup \cdots \cup S_{\eta+1}$ of $V(G)$ is called a \textit{$(k,d)$-partition} of $G$ if $|S_1| \leq k$, and $|S_j|=k$ for $j\in \{2, \ldots ,\eta+1\}$, and for each $j\in \{2, \ldots ,\eta+1\}$, there is such an ordering $\{x_1^j,\ldots ,x_k^j\}$ of vertices of $S_j$ that
\vspace{-0.15cm}
\begin{equation}
|N_G(x_i^j) \cap (S_1\cup\cdots \cup S_{j-1})| \leq di-1,\label{part-cond}
\end{equation}
for every $i\in \{1, \ldots ,k\}$.
Observe that if $S_1\cup \cdots \cup S_{\eta+1}$ is a $(k,d)$-partition of $G$, then $\eta+1=\left\lceil |V(G)|/k\right\rceil$.
Moreover, immediately by the definition, each  $(k,d)$-partition of $G$ is also its $(k,d+1)$-partition. Surprisingly, the monotonicity of the $(k,d)$-partition with respect to the parameter $k$ is not so easy to analyze.
 We illustrate this fact by Example \ref{ex:Gq}. Note that for integers $k,d$ the complexity of deciding whether $G$ has a $(k,d)$-partition is unknown.
The main result of this paper is as follows.
\begin{theorem}\label{thm:1}
Let $k,d,t\in \mathbb{N}$ and $t\geq k$. If a graph $G$ has  a $(k,d)$-partition, then it is equitably $(t,{\mathcal D}_{d-1})$-choosable. Moreover, there is a polynomial-time algorithm that for any graph with a given $(k,d)$-partition and for any $t$-uniform list assignment $L$ returns an equitable $(L,{\mathcal D}_{d-1})$-colouring of $G$.
\end{theorem}

The first statement of Theorem \ref{thm:1} generalizes the result obtained in \cite{DrFrDy18} for  $d\in \{1,2\}$. In this paper we  present an algorithm that confirms both, the first and second statements of Theorem \ref{thm:1} for all possible $d$. The algorithm, given in Section \ref{algor}, for a given $(k,d)$-partition of  $G$ with $t$-uniform list assignment $L$ returns its  equitable $(L,\mathcal{D}_{d-1})$-colouring.
Moreover, in Section \ref{grids} we give a polynomial-time algorithm that for a given 3-dimensional grid finds its $(3,2)$-partition, what, in consequence, implies $(t,\mathcal{D}_1)$-choosability of 3-dimensional grids for every $t\geq 3$. 

\section{The proof of Theorem \ref{thm:1}} \label{algor}
\subsection{Background}
For $S\subseteq V(G)$ by $G-S$ we denote a subgraph of $G$ induced by $V(G)\setminus S$. We start with a generalization  of some results given in \cite{Ko03,Pe04,Zh16} for classes ${\mathcal D}_0$ and ${\mathcal D}_1$. 
\begin{proposition}\label{l:2}
Let  $k,d\in \mathbb N$ and let $S$ be a set of distinct vertices $x_1, \ldots ,x_k$ of a graph $G$. If $G-S$ is equitably $(k,{\mathcal D}_{d-1})$-choosable and 
\vspace{-0.1cm}
$$|N_G(x_i)\setminus S|\le di-1$$
holds for every $i\in \{1, \ldots ,k\}$,
then $G$ is equitably $(k,{\mathcal D}_{d-1})$-choosable.
\end{proposition}
\begin{proof}
Let  $L$ be a $k$-uniform list assignment for $G$ and let  $c$ be an equitable $(L|_{V(G)\setminus S},{\mathcal D}_{d-1})$-colouring of $G-S$. Thus each colour class in $c$ has the cardinality at most $\left\lceil (|V(G)|-k)/k\right\rceil$ and induces in $G-S$, and consequently in $G$, a graph from ${\mathcal D}_{d-1}$. We extend $c$ to $(V(G)\setminus S) \cup \{x_k\}$ by assigning to $x_k$ a colour from $L(x_k)$ that is used on vertices in $N_G(x_k)\setminus S$ at most $d-1$ times. Such a colour always exists because 
$|N_G(x_k)\setminus S|\le dk-1 \hbox{ and } |L(x_k)|=k.$
Next, we colour vertices $x_{k-1}, \ldots, x_1$, sequentially, assigning to $x_i$ a colour from its list that is different from colours of all vertices $x_{i+1}, \ldots ,x_k$ and that is used at most $d-1$ times in $N_G(x_i)\setminus S$. Observe that there are at least $i$ colours in $L(x_i)$ that are different from $c(x_{i+1}), \ldots ,c(x_{k})$, and, since $|N_G(x_i)\setminus S|\le di-1$, $i\in \{1, \ldots ,k-1\}$, then such a choice of $c(x_i)$ is always possible. Next, the colouring procedure forces that 
 the cardinality of every colour class in the extended colouring $c$ is at most $\left\lceil |V(G)|/k\right\rceil$. Let $G_i=G[(V(G)\setminus S)\cup \{x_{i}, \ldots ,x_k\}]$. It is easy to see that for each $i\in \{1, \ldots ,k\}$ each colour class in $c|_{V(G_i)}$ induces a graph belonging to  ${\mathcal D}_{d-1}$. In particular this condition is satisfied for $G_1$, i.e. for $G$. Hence $c$ is an equitable $( L,{\mathcal D}_{d-1})$-colouring of $G$ and $G$ is equitably $(k,{\mathcal D}_{d-1})$-choosable.
 \hfill$\Box$
\end{proof}

Note that if a graph $G$ has a $(k,d)$-partition, then one can prove that $G$ is  equitably $(k,{\mathcal D}_{d-1})$-choosable by applying  Proposition \ref{l:2} several times. In general, the equitable $(k,{\mathcal D}_{d-1})$-choosability of $G$ does not imply the equitable $(t,{\mathcal D}_{d-1})$-choosability of $G$ for $t\geq k$. Unfortunately, if $G$ has a  $(k,d)$-partition, then $G$ may have  neither a $(k+1,d)$-partition nor  a $(k-1,d)$-partition.The infinite family of graphs defined in Example 1 confirms of the last fact.

\begin{example} Let $q\in \mathbb{N}$ and let ${G_1,\ldots,G_{2q+1}}$ be vertex-disjoint copies of $K_6$ such that $V(G_i)=\{v_1^i,\ldots ,v_6^i\}$ for $i\in \{1, \ldots 2q+1\}$. Let $G(q)$ $($cf. Fig. \ref{fig:exg2}$)$ be the graph resulted by  adding to ${G_1,\ldots,G_{2q+1}}$  edges that join  vertices of  $G_i$ with vertices of $G_{i-1}$, $i\in \{2,\ldots,2q+1\}$, in the following way:
$$
\begin{array}{ll}
 \mbox{for } i \mbox{ even:} &\mbox{for } i \mbox{ odd:} \\
  N_{G_{i-1}}(v_1^i)= \emptyset   &  N_{G_{i-1}}(v_1^i)=\{v_2^{i-1},v_3^{i-1},v_4^{i-1},v_5^{i-1},v_6^{i-1}\}\\
   N_{G_{i-1}}(v_2^i)=\{v_1^{i-1}\} & N_{G_{i-1}}(v_2^i)=\{v_1^{i-1},v_4^{i-1},v_5^{i-1},v_6^{i-1}\\
   N_{G_{i-1}}(v_3^i)=\{v_2^{i-1},v_3^{i-1}\}& N_{G_{i-1}}(v_3^i)=\{v_1^{i-1},v_2^{i-1},v_3^{i-1}\}\\
   N_{G_{i-1}}(v_4^i)=\{v_1^{i-1},v_2^{i-1},v_3^{i-1}\} & N_{G_{i-1}}(v_4^i)=\{v_2^{i-1},v_3^{i-1}\}\\
   N_{G_{i-1}}(v_5^i)=\{v_1^{i-1},v_4^{i-1},v_5^{i-1},v_6^{i-1}\} & N_{G_{i-1}}(v_5^i)=\{v_1^{i-1}\}\\
   N_{G_{i-1}}(v_6^i)=\{v_2^{i-1},v_3^{i-1},v_4^{i-1},v_5^{i-1},v_6^{i-1}\}\ & N_{G_{i-1}}(v_6^i)=\emptyset 
\end{array}
$$
\vspace{-0.3cm}
\begin{figure}[h]
    \centering
    \begin{tikzpicture}
  \matrix (net)
    [matrix of nodes,%
    nodes in empty cells, 
    nodes={inner sep=0pt,circle, fill, minimum size=4pt,draw},
    column sep={2.5cm,between origins},
    row sep={0.5cm,between origins}]
  {
      & & & & \\
      & & & & \\
      & & & & \\
      & & & & \\
      & & & & \\
      & & & & \\
  };
  \foreach \a in {1,...,6}
    \node[ left=0.2cm] at (net-\a-1) {$v_\a^1$};
  \foreach \a in {1,2,3,4,5}
    \node[left=0.2cm] at (net-1-\a) {$v_1^\a$};

  \foreach \a in {2,4,5}
    \node[right=0.2cm] at (net-6-\a) {$v_6^\a$};
  \node[left=0.2cm] at (net-6-3) {$v_6^3$};
   \foreach \a in {1,...,5}  
  \node[above right] at (net-1-\a) {$K_6$};
  \begin{scope}[on background layer]
    \foreach \a in {1,...,5}
    \node [rectangle,draw,rounded corners=0.65cm,inner sep=0.7cm,
      fit=(net-1-\a) (net-2-\a) (net-3-\a) (net-4-\a) (net-5-\a) (net-6-\a)] {};
  \end{scope};
  \draw (net-1-1) -- (net-2-2);
  \draw (net-1-1) -- (net-4-2);
  \draw (net-1-1) -- (net-5-2);
  \draw (net-2-1) -- (net-3-2);
  \draw (net-2-1) -- (net-4-2);
  \draw (net-2-1) -- (net-6-2);
  \draw (net-3-1) -- (net-3-2);
  \draw (net-3-1) -- (net-4-2);
  \draw (net-3-1) -- (net-6-2);
  \draw (net-4-1) -- (net-5-2);
  \draw (net-4-1) -- (net-6-2);
  \draw (net-5-1) -- (net-5-2);
  \draw (net-5-1) -- (net-6-2);
  \draw (net-6-1) -- (net-5-2);
  \draw (net-6-1) -- (net-6-2);
  \draw (net-1-2) -- (net-2-3);
  \draw (net-1-2) -- (net-3-3);
  \draw (net-1-2) -- (net-5-3);
  \draw (net-2-2) -- (net-1-3);
  \draw (net-2-2) -- (net-3-3);
  \draw (net-2-2) -- (net-4-3);
  \draw (net-3-2) -- (net-1-3);
  \draw (net-3-2) -- (net-3-3);
  \draw (net-3-2) -- (net-4-3);
  \draw (net-4-2) -- (net-1-3);
  \draw (net-4-2) -- (net-2-3);
  \draw (net-5-2) -- (net-1-3);
  \draw (net-5-2) -- (net-2-3);
  \draw (net-6-2) -- (net-1-3);
  \draw (net-6-2) -- (net-2-3);
  \draw (net-1-3) -- (net-2-4);
  \draw (net-1-3) -- (net-4-4);
  \draw (net-1-3) -- (net-5-4);
  \draw (net-2-3) -- (net-3-4);
  \draw (net-2-3) -- (net-4-4);
  \draw (net-2-3) -- (net-6-4);
  \draw (net-3-3) -- (net-3-4);
  \draw (net-3-3) -- (net-4-4);
  \draw (net-3-3) -- (net-6-4);
  \draw (net-4-3) -- (net-5-4);
  \draw (net-4-3) -- (net-6-4);
  \draw (net-5-3) -- (net-5-4);
  \draw (net-5-3) -- (net-6-4);
  \draw (net-6-3) -- (net-5-4);
  \draw (net-6-3) -- (net-6-4);
  \draw (net-1-4) -- (net-2-5);
  \draw (net-1-4) -- (net-3-5);
  \draw (net-1-4) -- (net-5-5);
  \draw (net-2-4) -- (net-1-5);
  \draw (net-2-4) -- (net-3-5);
  \draw (net-2-4) -- (net-4-5);
  \draw (net-3-4) -- (net-1-5);
  \draw (net-3-4) -- (net-3-5);
  \draw (net-3-4) -- (net-4-5);
  \draw (net-4-4) -- (net-1-5);
  \draw (net-4-4) -- (net-2-5);
  \draw (net-5-4) -- (net-1-5);
  \draw (net-5-4) -- (net-2-5);
  \draw (net-6-4) -- (net-1-5);
  \draw (net-6-4) -- (net-2-5);
\end{tikzpicture}
    \caption{A draft of graph $G(2)$ from Example \ref{ex:Gq}.}
    \label{fig:exg2}
\end{figure}
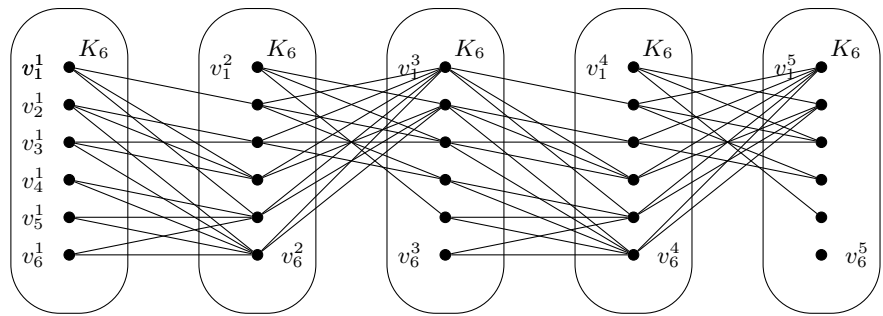

The construction of  $G(q)$ immediately implies that for every $q\in \mathbb{N}$ the graph $G(q)$ has a $(6,1)$-partition. Also, observe  that $deg_{G(q)}(v)\geq 5$ for each vertex $v$ of $G(q)$. Suppose that $G(q)$ has a $(5,1)$-partition $S_1\cup\cdots\cup S_{\eta+1}$  with $S_{\eta+1}=\{x_1^{\eta+1}, \ldots ,x_5^{\eta+1}\}$ such that $|N_{G(q)}(x_i^{\eta+1}) \cap (S_1\cup\cdots \cup S_{\eta})| \leq  i-1$. Thus $|N_{G(q)}(x_1^{\eta+1}) \cap (S_1\cup\cdots \cup S_{\eta})| = 0$ and consequently $deg_{G(q)}(x_1^{\eta+1})\le 4$, contradicting our previous observation. Hence, $G(q)$ has no $(5,1)$-partition. 
Next, we will show that $G(q)$ has no $(7,1)$-partition for $q\geq 2$. 

\begin{proposition}\label{prop:1}
For every integer $q$ satisfying $q\geq 2$ the graph $G(q)$ constructed in {\rm Example \ref{ex:Gq}}  has no $(7,1)$-partition.
\end{proposition}

\begin{proof}
Let $G=G(q)$ and suppose that $S^\prime_1\cup \cdots \cup S^\prime_{\eta+1}$ is a $(7,1)$-partition of $G$. 
  Let  $x_1^{\eta+1},x_2^{\eta+1},\ldots, x_7^{\eta+1}$ be an ordering of vertices of $S^\prime_{\eta+1}$ such that $|N_G(x_i^{\eta+1}) \cap (S^\prime_1\cup\cdots \cup S^\prime_{\eta})| \leq  i-1$. Thus  $\deg_G(x_1^{\eta+1})\le 6$. Since for every vertex $v\in V(G)\setminus \{v_5^{2q+1},v_6^{2q+1}\}$ it holds that $\deg_G(v)\ge 7$, we have $x_1^{\eta+1}=v_6^{2q+1}$ or $x_1^{\eta+1}=v_5^{2q+1}$ and furthermore $N_G(x_1^{\eta+1})\subseteq S^\prime_{\eta+1}$.  So, $\{v_1^{2q+1},\ldots,v_6^{2q+1}\}\subseteq S^\prime_{\eta+1}$ and we must recognize the last vertex of $S^\prime_{\eta+1}$. Since every vertex in $S^\prime_{\eta+1}$ has at most $6$ neighbors outside $S^\prime_{\eta+1}$, it follows that either $v_1^{2q}\in S^\prime_{\eta+1}$ or $v_2^{2q}\in S^\prime_{\eta+1}$.
Let $G^\prime=G-S^\prime_{\eta+1}$, $S^\prime_{\eta}=\{x_1^{\eta},x_2^{\eta},\ldots, x_7^{\eta}\}$, and $|N_{G^\prime}(x_i^{\eta}) \cap (S^\prime_1\cup\cdots \cup S^\prime_{\eta-1})| \leq  i-1$.

\noindent {\it Case} 1. $v_1^{2q}\in S^\prime_{\eta+1}$

Since $\deg_{G^\prime}(x_1^{\eta})\le 6$, we have that either $x_1^{\eta}=v_2^{2q}$ or $x_1^{\eta}=v_3^{2q}$. In the first case,  $\{v_1^{2q-1},v_2^{2q},\ldots,v_6^{2q}\}\subseteq S^\prime_{\eta}$ and we must recognize the last vertex of $S^\prime_{\eta}$. However, $v_1^{2q-1}$ has 10 neighbors outside $\{v_1^{2q-1},v_2^{2q},\ldots,v_6^{2q}\}$. It contradicts the condition $v_1^{2q-1}\in S^\prime_{\eta}$, since every vertex in $S^\prime_{\eta}$ can have at most 6 neighbors outside.  In the second case, $\{v_2^{2q-1},v_3^{2q-1},v_2^{2q},\ldots,v_6^{2q}\}= S^\prime_{\eta}$. Similarly, we can observe that $v_2^{2q-1}$ has $8$ neighbors outside $S^\prime_{\eta}$, which contradicts the fact $v_2^{2q-1}\in S^\prime_{\eta}$.

\noindent {\it Case} 2. $v_2^{2q}\in S^\prime_{\eta+1}$

Since $\deg_{G^\prime}(x_1^{\eta})\le 6$, we have that either $x_1^{\eta}=v_1^{2q}$ or $x_1^{\eta}=v_3^{2q}$. By the same argument as in \emph{Case} $1$ we can observe that $x_1^{\eta}\neq v_3^{2q}$, and so we may assume that $x_1^{\eta}=v_1^{2q}$ and $\{v_1^{2q},v_3^{2q},\ldots ,v_6^{2q}\}\subseteq S^\prime_{\eta}$. Thus we  must recognize two last vertices of $S^\prime_{\eta}$. Since each of $v_1^{2q-1},v_2^{2q-1},v_3^{2q-1}$ has at least $8$ neighbors outside $\{v_1^{2q},v_3^{2q},\ldots ,v_6^{2q}\}$, we have $\{v_1^{2q-1},v_2^{2q-1},v_3^{2q-1}\}\cap S^\prime_{\eta}=\emptyset $.  Hence we  conclude that $v^{2q}_3,v^{2q}_4,v^{2q}_5,v^{2q}_6$ have at least two neighbors outside $S^\prime_{\eta}$, and so no of them is $x_2^{\eta}$, because $x_2^{\eta}$ has at most one neighbor  outside $S^\prime_{\eta}$. Furthermore, every vertex in $V(G^\prime)\setminus \{v_1^{2q},v_3^{2q},\ldots ,v_6^{2q}\}$ has at least $5$ neighbors outside $\{v_1^{2q},v_3^{2q},\ldots ,v_6^{2q}\}$, thus also none of them is $x_2^{\eta}$. Since $x_2^{\eta}$ does not exist, the $(7,1)$-partition of $G$ does not exist either.
\end{proof}
\label{ex:Gq}
\end{example}

\subsection{Algorithm}\label{subal}
Now we are ready to present the algorithm that confirms both statements of Theorem \ref{thm:1}. Note that Proposition \ref{l:2} and the induction procedure could be used to prove the first statement of Theorem \ref{thm:1} for $t=k$ but   this approach  seems to be useless for $t> k$, as we have observed in Example \ref{ex:Gq}. 
Mudrock et al. \cite{mudrock} proved the lack of monotonicity for the equitable $(k,\mathcal{D}_0)$-choosability with respect to the parameter $k$.  It motivates our approach for solving the problem.
To simplify understanding we give the main idea of the algorithm  presented in the further part. It can be expressed in a few steps (see also Fig. \ref{fig:my_label}):
\begin{itemize}
    \item on the base of the given $(k,d)$-partition $S_1\cup \cdots \cup S_{\eta+1}$ of $G$, we create the list $S$, consisting of the elements from $V(G)$, whose order corresponds to the order in which the colouring is expanded to successive vertices in each $S_j$ (cf. the proof of Proposition \ref{l:2});
    \item let $|V(G)|=\beta t +r_2$, $1 \leq r_2 \leq t$; we colour $r_2$ vertices from the beginning of $S$ taking into account the lists of available colours; we delete the colour assigned to $v$ from the lists of available colours for vertices from $N^{col}_G(d,v)$;
    \item let $|V(G)|=\beta (\gamma k +r)+r_2=\beta \gamma k+ (\rho k +x) +r_2$; observe that $r \equiv t \pmod k$ and $\rho k +x\equiv 0 \pmod r$; we colour $\rho k +x$ vertices taking into account the lists of available colours in such a way that every sublist of length $k$ is formed by vertices coloured differently (consequently, every sublist of length $r$ is coloured differently); we divide the vertices colored here into $\beta$ sets  each one of cardinality $r$; we delete $c(v)$ from the lists of vertices from $N^{col}_G(d,v)$;
    \item we extend the list colouring into the uncoloured $\beta\gamma k$ vertices by colouring $\beta$ groups of $\gamma k$ vertices; first, we associate each group of $\gamma k$ vertices with a set of $r$ vertices coloured  in  the previous step (for different groups these sets are disjoint); next, we color the vertices of each of the group using  $\gamma k$ different colors that are also different from the colors of $r$ vertices of the set associated with this group;
    \item our final equitable list colouring is the consequence of a partition of $V(G)$ into $\beta+1$ coloured sets, each one of size at most $t$ and each one formed by vertices coloured differently.
    \end{itemize}
   
\begin{figure}[htb]
    \centering
   \begin{tikzpicture}
  [x=1.2cm,y=1cm,ovals/.style={rectangle,draw,rounded corners=0.25cm,minimum width=0.5cm,minimum height=2cm}]
  \node[ovals,minimum height=1cm] at (0,-0.5) {$r_1$};
  \node[ovals] at (1,0) {$k$};
  \node[ovals] at (2,0) {$k$};
  \node at (3.5,0) {$\dots$};
  \node[ovals] at (5,0) {$k$};
  \node[ovals] at (6,0) {$k$};
  \node at (7.5,0) {$\dots$};
  \node[ovals] at (9,0) {$k$};
  \node at (0,-2) {$S_1$};
  \node at (2.5,-2) {$\dots$};
  \node at (5,-2) {$S_{\eta-\beta\gamma+1}$};
  \node at (6.1,-2) {$S_{\eta-\beta\gamma+2}$};
  \node at (7.5,-2) {$\dots$};
  \node at (9,-2) {$S_{\eta+1}$};
%
  \draw (5.5,-2) -- (5.5,2);
  \draw[rounded corners=0.35cm] (-0.5,-1.25) -- (2.5,-1.25) --
    (2.5,-0.25) -- (1.4,-0.05) -- (1.4,1.25) -- (0.6,1.25) --
    (0.6,0.45) -- (-0.5,0.20) -- cycle;
  \draw[rounded corners=0.35cm] (2.7,-1.25) -- (5.35,-1.25) --
    (5.35,1.25) -- (1.6,1.25) -- (1.6,0.25) -- (2.7,0.05) -- cycle;
  \node at (3.5,1.5) {$(\rho k+x)$ vertices};
  \node[rotate=40] at (0,0.85) {$r_2$ vertices};
\end{tikzpicture}
    \caption{An exemplary illustration of the input of the \textsc{Equitable $(L,{\mathcal D}_{d-1})$-colouring} algorithm.}
    \label{fig:my_label}
\end{figure}
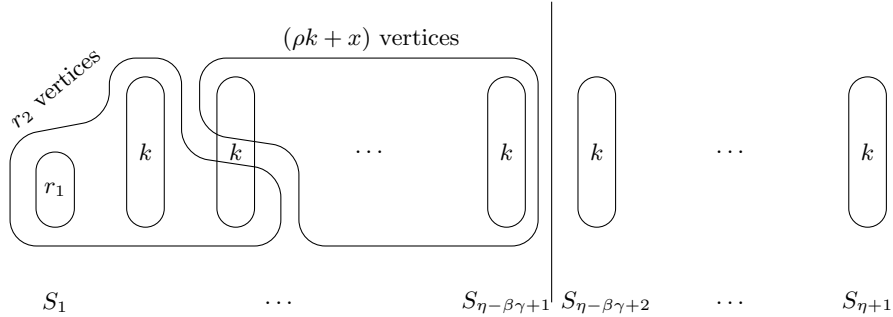

\SetKwInOut{Input}{Input}
\SetKwInOut{Output}{Output}
\setlength{\intextsep}{0pt}
\begin{algorithm}\label{al:1}
\caption{\textsc{Equitable $(L,{\mathcal D}_{d-1})$-colouring}($G$)}
\Input {Graph $G$ on $n$ vertices;  $L$ - $t$-uniform list assignment; a  $(k,d)$-partition $S_1\cup \cdots \cup S_{\eta+1}$ of $G$, given by lists $S_1=(x^1_1,\ldots,x^1_{r_1})$ and $S_j=(x^j_1,\ldots,x^j_k)$ for $j\in \{2, \ldots ,\eta+1\}$.}
\Output{Equitable $(L,{\mathcal D}_{d-1})$-colouring of $G$.}
initialization\;
$S:=empty$; $L_R:=empty$; $L_X:=empty$;\\
\For{$j:=1$ \textbf{to} $\eta+1$}{ 
add \textsc{reverse}$(S_j)$ to $S$; 
\hfill//\textsc{reverse} is the procedure for reversing lists\label{procedureReverse}}
$\beta:=\left\lceil n/t\right\rceil$-1;\\
\eIf {$n \equiv0 \pmod t$}{$r_2:=t$}{$r_2:=n \pmod t;$}
$\gamma := t \div k$; $r:= t \pmod k$;
$\rho:=\beta r \div k$; $x:=\beta r \pmod k$;\\
take and delete $r_2$ elements from the beginning of $S$, and add them, vertex \\
\hspace{2cm} by vertex, to list $L_R$; \label{r2}\\
\textsc{colour\_List}$(L_R,r_2)$; \label{colour-list-one}\\
take and delete $x$ elements from the beginning of $S$, and add them, vertex\\
\hspace{2cm} by vertex, to list $L_X$ ;\\
\textsc{colour\_List}$(L_X,x)$; \label{colour-list-two}\\
$S_{col}:=L_X$;\\
\For{$j=1$ \textbf{to} $\rho$}{\label{procedureOne-for-one}
take and delete $k$ elements from the beginning of $S$, and add them, vertex\\ 
\hspace{2cm} by vertex, to list $S'$;\\
\textsc{colour\_List}$(S',k)$; \label{colour-list-three}\\
$S_{col}:=S_{col} + S'$;}
\textsc{Reorder}($S_{col}$); \label{reorder}\\
$\overline{S}:=S$; \hfill //an auxiliary list\\ \label{modify-colour-one}
\textsc{Modify\_colourLists}$(S_{col},\overline{S})$; \label{modify-list}\\
\textsc{colour\_List}$(S,\gamma k)$;\label{colour-list-four}\\
\end{algorithm}
\setlength{\intextsep}{0pt}
\begin{algorithm}\renewcommand{\algorithmcfname}{Procedure}
\caption{\textsc{colour\_List}($S',p$)}\label{colour-list}
\Input {List $S'$ of vertices; integer $p$.//the length of $S'$ is multiple of $p$}
\Output {$L$-colouring of the vertices from $S'$.\\//The procedure also modifies a global variable of the list assignment $L$.
}
initialization\;
\While{$S'\neq empty$}{\label{procedureTwo-while-one}
let $S''$ be the list of the $p$ first elements of $S'$;\\
$C:=\emptyset$; \hfill //set $C$ is reserved for the colours being assigned to vertices of $S''$\\
\While{$S''\neq empty$}{\label{procedureTwo-while-two}
let $v$ be the first element of $S''$;\label{vdef}\\
$L(v):=L(v) \backslash C$;\\
$c(v):=$\textsc{colour\_Vertex}$(v)$; \\
delete vertex $v$ from $S'$ and $S''$; 
$C:=C\cup \{c(v)\}$;\\
}
}

\end{algorithm}
\setlength{\intextsep}{0pt}
\begin{algorithm}\renewcommand{\algorithmcfname}{Procedure}
\caption{\textsc{colour\_Vertex}($v$)} \label{colour-vertex}
\Input{Vertex $v$ of the graph $G$.}
\Output{$L$-colouring of the vertex $v$.\\//The procedure modifies also a global variable of the list assignment $L$.}
initialization\;
$c(v):=$ any colour from $L(v)$;\\
delete $c(v)$ from $L(w)$ for all $w \in N_G^{col}(d,v)$;\label{delete} \hfill //$d$ is a global variable\\
\textbf{return} $c(v)$;\\
\end{algorithm}
\setlength{\intextsep}{0pt}
\begin{algorithm}\renewcommand{\algorithmcfname}{Procedure}
\caption{\textsc{Reorder}($S'$)}
\Input {List $S'$ of coloured vertices of $G$.}
\Output{List $S'$ - reordered in such a way that every its sublist of length $k$ is formed by vertices being coloured with different colours.}
initialization\;
$S_{aux}:=empty$; \hfill//an auxiliary list\\
take and delete $r_1$ elements from $S'$, and add them, vertex by vertex, to $S_{aux}$;\\
\For{$j=1$ \textbf{\emph{to}} $\eta-\beta \gamma$}{
$P:=\emptyset$;\\
take and delete first $k$ elements from $S'$, and add them to set $P$;\\
\For{$i=1$ \textbf{\emph{to}} $k$ \textbf}{
let $v$ be a vertex from $P$ such that $c(v)$ is different from colours of the last $k-1$ vertices of $S_{aux}$; 
add $v$ to the end of $S_{aux}$;\\
}
}
$S':=S_{aux}$;\\
\end{algorithm}
\setlength{\intextsep}{0pt}
\begin{algorithm}\renewcommand{\algorithmcfname}{Procedure}
\caption{\textsc{Modify\_colourLists}($L_1,L_2$)}\label{modify-colour}
\Input {List $L_1$ of $\beta r$ coloured vertices and list $L_2$ of $\beta \gamma k$ uncoloured vertices.}
\Output {Modified colour list assignment $L$ for vertices of $L_2$.\\ //$L$ is a global variable
}
initialization\;
$C:=\emptyset$;\hfill //$C$ is a set of colours of vertices from the depicted part of $L_1$\\
\For{$i=1$ \textbf{\emph{to}} $\beta$}{
take and delete first $r$ vertices from $L_1$;\\ let $C$ be the set of colours assigned to them;\\
\For{$j=1$ \textbf{\emph{to}} $\gamma k$}{
let $v$ be the first vertex from $L_2$;\\
$L(v):=L(v)\backslash C$;
delete $v$ from $L_2$;\\
}
}
\end{algorithm}
Now we illustrate \textsc{Equitable $(L,{\mathcal D}_{d-1})$-colouring} using a graph constructed in Example \ref{ex:1}.
\begin{example}\label{ex:1} Let $G_1,G_2$ be two vertex-disjoint copies of $K_5$ and $V(G_i)=\{v_1^i,\ldots,$ $v_5^i\}$ for $i\in \{1,2\}$.  We join every vertex $v_j^2$ to $v_j^1,v_{j+1}^1, \ldots, v_5^1$ for $j\in\{1,2,3,4,5\}$. Next, we add a vertex $w^i_j$ and join it with $v^i_j$ for $i\in\{1,2\}\;j\in\{1,2,3,4,5\}$. In addition, we join $w^i_j$ to arbitrary two  vertices in $\{v^q_p:(q<i)\vee (q=i\wedge p<j)\}\cup \{w^q_p:(q<i)\vee (q=i\wedge p<j)\}$, $i\in\{2,3,4,5\}\;j\in\{1,2\}$. Let $G$ be a resulted graph. Observe that $|V(G)|=20$ and the  partition $S_1\cup S_2\cup\ldots \cup S_{10}$ of $V(G)$ such that $S_{p+1}$$=\{v^{s+1}_{r+1},w^{s+1}_{r+1}\}$ for $p\in \{0,\ldots, 9\}$, where $s=\left\lfloor \frac{p}{5}\right\rfloor, r\equiv p \pmod 5$ is a  $(2,3)$-partition of $G$. 
\begin{figure}
    \centering
\begin{tikzpicture}
  \matrix (net)
    [matrix of nodes,%
    nodes in empty cells,
    nodes={outer sep=0pt,circle,minimum size=4pt,draw},
    column sep={2.5cm,between origins},
    row sep={0.6cm,between origins}]
  {
      & & &  \\
      & & &  \\
      & & &  \\
      & & &  \\
      & & &  \\
  };
  \foreach \a in {1,...,5}
    \node[ left=0.2cm] at (net-\a-1) {$w_\a^1$};
   \foreach \a in {1,...,5}
    \node[ right=0.2cm] at (net-\a-4) {$w_\a^2$};
 \foreach \a in {1,...,5}
    \node[below left=0.05cm] at (net-\a-2) {$v_\a^1$};
    \foreach \a in {1,...,5}
    \node[below right=0.05cm] at (net-\a-3) {$v_\a^2$};
    
  \node[above=0.2cm] at (net-1-2) {$K_5$};
  \node[above=0.2cm] at (net-1-3) {$K_5$};
 
 \begin{scope}[on background layer]
    \foreach \a in {2,3}
    \node [rectangle,draw,rounded corners=0.65cm,inner sep=0.6cm,
      fit=(net-1-\a) (net-2-\a) (net-3-\a) (net-4-\a)
      (net-5-\a)] {};
  \end{scope};
  \draw (net-1-1) -- (net-2-1);
  \draw (net-2-1) -- (net-3-1);
  \draw (net-3-1) -- (net-4-1);
  \draw (net-4-1) -- (net-5-1);
%
  \draw (net-1-4) -- (net-2-4);
  \draw (net-2-4) -- (net-3-4);
  \draw (net-3-4) -- (net-4-4);
  \draw (net-4-4) -- (net-5-4);
\foreach \a in {1,...,5}
  \draw (net-\a-2) -- (net-\a-1);
 \foreach \a in {1,...,5}
  \draw (net-\a-3) -- (net-\a-4); 
 \draw (net-1-2) -- (net-2-1);
 \draw (net-2-2) -- (net-3-1);
 \draw (net-3-2) -- (net-4-1);
 \draw (net-4-2) -- (net-5-1);
 \draw (net-2-4) -- (net-1-3);
 \draw (net-3-4) -- (net-2-3);
 \draw (net-4-4) -- (net-3-3);
 \draw (net-5-4) -- (net-4-3);
%
%
\foreach \b in {1,...,5}
 \foreach \a in {\b,...,5}
   \draw (net-\b-3) -- (net-\a-2);
\node at (net-1-1) {$2$};
\node at (net-2-1) {$2$};
\node at (net-3-1) {$2$};
\node at (net-4-1) {$3$};
\node at (net-5-1) {$3$};
 \node at (net-1-2) {$1$};
 \node at (net-2-2) {$1$};
 \node at (net-3-2) {$1$};
 \node at (net-4-2) {$2$};
 \node at (net-5-2) {$2$};
 \node at (net-1-3) {$3$};
 \node at (net-2-3) {$2$};
 \node at (net-3-3) {$1$};
 \node at (net-4-3) {$1$};
 \node at (net-5-3) {$1$};
 \node at (net-1-4) {$1$};
 \node at (net-2-4) {$4$};
 \node at (net-3-4) {$4$};
 \node at (net-4-4) {$2$};
 \node at (net-5-4) {$4$};
\end{tikzpicture}
    \caption{An exemplary graph $G$ depicted in Example \ref{ex:1} with an exemplary colouring returned by \textsc{Equitable $(L,{\mathcal D}_{d-1})$-colouring}($G$).}
    \label{ex:alg}
\end{figure}

For the purpose of Example \ref{ex:1}, we assume the following 3-uniform list assignment for the graph from Figure \ref{ex:alg}: $L(v^i_j)=\{1,2,3\}$, for $i\in \{1,2,3,4\}$, $L(w^1_j)=\{2,3,4\}$, and $L(w^2_j)=\{1,2,4\}$, $j\in \{1,2,3,4,5\}$, while given $(2,3)$-partition of $G$ is:  $S_p=\{w^{s+1}_{r+1},v^{s+1}_{r+1}\}$, where $s=\left\lfloor \frac{p}{5}\right\rfloor, r\equiv p \pmod 5$, $p\in[10]$. 

Thus \textsc{Equitable $(L,{\mathcal D}_{d-1})$-colouring} returns equitable $(L,\mathcal{D}_2)$-coloring of $G$. Note, that $20=|V(G)|=\eta \cdot k+r_1=9\cdot 2 +2$. While on the other hand, we have $20=|V(G)|=\beta \cdot t +r_2=6 \cdot 3+2$. Observe that $x=0$. When we colour a vertex, we always choose the first colour on its list.
The list $S$ determined in lines 3-5 of \textsc{Equitable $(L,{\mathcal D}_{d-1})$-colouring} and the colours assigned to first part of vertices of $S$ (lines  18-25) are as follows:
$$
\begin{array}{rc|c|c}
     S=&(v_1^1,w_1^1,& v_2^1,w_2^1,v_3^1,w_3^1,v_4^1,w_4^1,&v_5^1,w_5^1,v_1^2,w_1^2,\ldots,v_5^2,w_5^2)  \\
     &r_2 &\rho k&\beta \gamma k\\ 
     \mbox{\emph{ colours:}} & 1 \ \ \ 2 & 1\ \ \ 2\ \ \ 1^*\ \ \ 2\ \ \ 2\ \ \ 3\ \ & \\
    \end{array}
$$
\noindent $^*$: after colouring $v_3^1$ with 1, $L(v_1^2)=\{2,3\}$ - the result of line \ref{delete} in the \textsc{colour\_Ver\-tex} procedure. 

List $S_{col}$ after \textsc{Reorder}{($S_{col}$)}: $(v_2^1,w_2^1,v_3^1,w_3^1,w_4^1,v_4^1)$ with corresponding colours: $(1,2,1,2,3,2)$.

$$
\begin{array}{*{13}c}
   \overline{S}= & (v_5^1, &w_5^1, & v_1^2, & w_1^2,& v_2^2, & w_2^2,& v_3^2, & w_3^2,& v_4^2, & w_4^2,& v_5^2, &w_5^2) \\\hline
\rm{lists\ after} & 2& 2& 3 & 1& 2& 2& 1& 1& 1& 1& 1& 1 \\
\rm{procedure} & 3 & 3 &  & 4 &3 & 4 & 3 & 4 & 2 & 2 & 3 & 4 \\
\textsc{Modify\_colourList}& & 4 & & & &  & & & & 4 & & \\\hline
{\rm final}\ c(v) & 2 & 3 & 3 & 1 & 2 & 4 & 1 & 4 & 1 & 2 & 1&4\\
\end{array}
$$
\end{example}

To prove the correctness of the \textsc{Equitable $(L,{\mathcal D}_{d-1})$-colouring} algorithm,  we give some observations and  lemmas.
\begin{observation}\label{obs:al}
The colour function $c$ returned by \textsc{Equitable $(L,{\mathcal D}_{d-1})$-col\-ouring} is constructed step by step. In each step, $c(v)$ is a result of {\rm \textsc{colour\_Ver\-tex}{($v$)}} and this value is not changed further. 
\end{observation}
\begin{observation}\label{obs:al:1}
The list assignment $L$, as a part of the input of \textsc{Equitable $(L,{\mathcal D}_{d-1})$-colouring}, is modified for a vertex $v$ by \textsc{colour\_List}  or by \textsc{Modify\_colourLists}.
\end{observation}

\begin{lemma}\label{lem:1} Every time when \textsc{Equitable $(L,{\mathcal D}_{d-1})$-colouring}$(G)$ calls  {\rm \textsc{col\-our\_Ver\-tex}{($v$)}}, the list $L(v)$ for $v$ is non-empty, i.e. {\rm \textsc{colour\_Vertex}{($v$)}} is always executable.
\end{lemma}

\begin{proof}
Note that  \textsc{colour\_Vertex} is called by \textsc{colour\_List}.  
Let 

\noindent $R=\{v\in V(G): \textsc{colour\_Vertex}{(v)} \mbox{ is called when }\; \textsc{colour\_List}(L_R,r_2) \linebreak \mbox{ in line \ref{colour-list-one}}\;   \mbox{of }\;  \textsc{Equitable} (L,{\mathcal D}_{d-1})\textsc{-colouring}  \mbox{ is executed}\}$, 

\noindent $X=\{v\in V(G): \textsc{colour\_Vertex}{(v)}  \mbox{ is called when }\; \textsc{colour\_List}(L_X,x) \linebreak  \mbox{ in line \ref{colour-list-two}}\;   \mbox{of }\;  \textsc{Equitable} (L,{\mathcal D}_{d-1})\textsc{-colouring}  \mbox{ is executed}\}$.

\noindent Let $V_1:=S_1\cup \cdots \cup S_{\eta+1-\beta\gamma}$, $V_2:=V(G)\setminus V_1=S_{\eta+1-(\beta\gamma-1)}\cup \cdots \cup S_{\eta+1}$. Note that

\noindent $V_1\setminus (R \cup X)=\{v\in V(G):$ \textsc{colour\_Vertex}${(v)}$ is called when 
\textsc{colour\_List}$(S',k)$ in line \ref{colour-list-three} of \textsc{Equitable} $(L,{\mathcal D}_{d-1})$
\textsc{-colouring} is executed$\}$

\noindent $V_2=\{v\in V(G):$ \textsc{colour\_Vertex}${(v)}$ is called when  \textsc{colour\_List}$(S,\gamma k)$ in line \ref{colour-list-four} of \textsc{Equitable} $(L,{\mathcal D}_{d-1})$\textsc{-colouring}  is executed$\}$.

\noindent Observe that $|R|=r_2,|X|=x, |V_1\setminus (R\cup X)|=\rho k, |V_2|=\beta\gamma k$.

\noindent {\it Case} 1. $v\in  R$. 

In this case, the vertex $v$ is coloured by \textsc{colour\_List}{($L_R,r_2$)} in line \ref{colour-list-one} of \textsc{Equitable $(L,{\mathcal D}_{d-1})$-colouring}. Since $|R|=r_2$, the {\bf while} loop  in line \ref{procedureTwo-while-one} of \textsc{colour\_List} is executed only once. The {\bf while} loop in line \ref{procedureTwo-while-two} of  \textsc{colour\_List} is executed $r_2$ times. Suppose, $v$ is a vertex such that \textsc{colour\_Vertex}{($v$)} is called in the $i$-th execution of the {\bf while} loop in line \ref{procedureTwo-while-two} of \textsc{colour\_List}. By Observation \ref{obs:al:1}, the fact that \textsc{Equitable $(L,{\mathcal D}_{d-1})$-colouring} has not  called \textsc{Modify\_colourLists} so far, and because it is the first time when \textsc{colour\_List} works, we have $|C|=i-1$, and $L(v)\setminus C$ is the current list of $v$. Since $t\ge r_2$, the list of  $v$ is non-empty.

\noindent {\it Case} 2. $v\in X$.

This time, the vertex $v$ is coloured by \textsc{colour\_List}{($L_X,x$)} called in line \ref{colour-list-two} of \textsc{Equitable $(L,{\mathcal D}_{d-1})$-colouring}. Similarly as in \emph{Case} 1, the {\bf while} loop  in line \ref{procedureTwo-while-one} of \textsc{colour\_List} is executed only once and the {\bf while} loop in line \ref{procedureTwo-while-two} of \textsc{colour\_List} is executed $x$ times. Suppose that $v$ is a vertex such that \textsc{colour\_Vertex}{($v$)} is called in the $i$-th  iteration of the {\bf while} loop in line \ref{procedureTwo-while-two} of \textsc{colour\_List}.
Observe that properties of the  $(k, d)$-partition $S_1\cup \cdots \cup S_{\eta+1}$ of $G$ (given as the input of  \textsc{Equitable $(L,{\mathcal D}_{d-1})$-colouring}) and the \textsc{Reverse} procedure from line \ref{procedureReverse} of \textsc{Equitable $(L,{\mathcal D}_{d-1})$-colouring} imply that $v$ has at most $(x-i+1)d-1+(k-x)$ neighbors $w$ for which  \textsc{colour\_Vertex}{($w,d$)} was executed earlier than \textsc{colour\_List}{($L_X,x$)}.  More precisely, by the definition of the $(k, d)$-partition, $v$ has at most  $(x-i+1)d-1$ neighbours in $R\setminus Y$, where $Y$ consists of the last  $k-x$  vertices $w$ for which  \textsc{colour\_Vertex}{($w$)} was executed (being called  by \textsc{colour\_List}{($L_R,r_2$)}).
Thus, at most $k-i$ colours were deleted from $L(v)$ before \textsc{colour\_List}{($L_X,x$)}  began. If the {\bf while} loop in line \ref{procedureTwo-while-two} of \textsc{colour\_List} is called for the $i$-th time, then $|C|=i-1$ and so, from the current list $L(v)$ at most $i-1$ elements were deleted. Furthermore, \textsc{Equitable $(L,{\mathcal D}_{d-1})$-colouring} has not called  \textsc{Modify\_colourLists} so far. Thus the current size of $L(v)$ is at least $t-k+1$, by Observation \ref{obs:al:1}. Since $t\geq k$, the list of $v$ is non-empty.

\noindent {\it Case} 3. $v\in V_1\setminus (R\cup X)$. 

In this case, the vertex $v$ is coloured during the execution of the {\bf for} loop in line \ref{procedureOne-for-one} of \textsc{Equitable $(L,{\mathcal D}_{d-1})$-colouring}. Observe that this loop is executed $\rho$ times, and
in each of the executions the {\bf while} loop  in line \ref{procedureTwo-while-one} of \textsc{colour\_List} is executed only once, while the {\bf while} loop in line \ref{procedureTwo-while-two} of \textsc{colour\_List} is executed $k$ times. 

Suppose that $v$ is a vertex such that \textsc{colour\_Vertex}{($v$)} is called in the $j$-th execution of the {\bf for} loop in line \ref{procedureOne-for-one} of \textsc{Equitable $(L,{\mathcal D}_{d-1})$-colouring}  and  the $i$-th execution of the {\bf while} loop in line \ref{procedureTwo-while-two} of \textsc{colour\_List}. Observe that properties of the  $(k, d)$-partition $S_1\cup \cdots \cup S_{\eta+1}$ of $G$ (given as input of   \textsc{Equitable $(L,{\mathcal D}_{d-1})$-colouring}) and the  \textsc{Reverse} procedure from line \ref{procedureReverse} of \textsc{Equitable $(L,{\mathcal D}_{d-1})$-colouring} imply that $v$ has at most $(k-i+1)d-1$  neighbors $w$ for which \textsc{colour\_Vertex}{($w$)} was executed before the $j$-th execution of the {\bf for} loop in line \ref{procedureTwo-while-one} of  \textsc{Equitable $(L,{\mathcal D}_{d-1})$-colouring} started. Thus from the list of $v$ at most $k-i$ colours were deleted before the $j$-th execution of the {\bf for} loop in line \ref{procedureTwo-while-one} of \textsc{Equitable $(L,{\mathcal D}_{d-1})$-colouring}  started. If the {\bf while} loop in line \ref{procedureTwo-while-two} of \textsc{colour\_List} is called for the $i$-th time, $|C|=i-1$ and so, from the current list of $v$ at most $i-1$ elements were deleted. Furthermore, \textsc{Equitable $(L,{\mathcal D}_{d-1})$-colouring} has not  called  \textsc{Modify\_colourLists} so far. Thus the current size of the list $L(v)$ is at least $t-k+1$, by Observation \ref{obs:al:1}. Since $t\geq k$, the list of $v$ is non-empty.

\noindent {\it Case} 4.  $v\in V_2$. 

In this case, the vertex $v$ is coloured by the \textsc{colour\_List}{($S,\gamma k$)} procedure, called by \textsc{Equitable $(L,{\mathcal D}_{d-1})$-colouring} in line \ref{colour-list-four}.
Since $|V_2|=\beta\gamma k$, the {\bf while} loop in line \ref{procedureTwo-while-one} of \textsc{colour\_List} is executed $\beta$ times and the {\bf while} loop in line \ref{procedureTwo-while-two} of \textsc{colour\_List} is executed $\gamma k$ times. Suppose that $v$ is a vertex  for which  \textsc{colour\_Vertex}{($v,d$)} is called in the $j$-th execution of the {\bf while} loop in line \ref{procedureTwo-while-one}  of \textsc{colour\_List} and in the $z$-th execution of the {\bf while} loop in line \ref{procedureTwo-while-two} of \textsc{colour\_List}. Let  $z=yk+i$, where $0\le y\le \gamma -1,\; 0\le i\le k$.

Similarly as before, properties of the  $(k, d)$-partition $S_1\cup \cdots \cup S_{\eta+1}$ of $G$ and the \textsc{Reverse} procedure  imply that $v$ has at most $(k-i+1)d-1$  neighbors $w$ for which \textsc{colour\_Vertex}{($w$)} was executed before the $j$-th execution of the {\bf while} loop in line \ref{procedureTwo-while-one} of  \textsc{colour\_List} began. Thus from the list of $v$ at most $k-i$ elements were deleted before the $j$-th execution of the {\bf while} loop in line \ref{procedureTwo-while-one} of \textsc{colour\_List}  started. Furthermore, after the execution of  \textsc{Modify\_colourLists} in line \ref{modify-colour-one} from the list of every vertex in $V_2$ at most $r$ elements were  removed.  
If the {\bf while} loop in line \ref{procedureTwo-while-two} of \textsc{colour\_List} is called for the $z$-th time, then $|C|=z-1$ and so, from the initial  list $L(v)$ at most $z-1$ elements were deleted. Thus, when \textsc{colour\_Vertex}{($v$)} is called, then the size of the current list $L(v)$ is at least $t-r-(k-i)-(z-1)=t-r-k-yk+1$. Since $y\le \gamma-1$ and $t=\gamma k+r$, the list of $v$ is non-empty.
\hfill$\Box$
\end{proof}

\begin{lemma}\label{com:col}
An output of {\rm \textsc{Equitable $(L,{\mathcal D}_{d-1})$-colouring}($G$)}  is an $(L,{\mathcal D}_{d-1})$-colouring of $G$.
\end{lemma}

\begin{proof}
We will show that if \textsc{colour\_Vertex}{($v$)} is executed, then an output $c(v)$ has always the following property. For each subgraph  $H$ of $G$ induced by vertices $x$ for which \textsc{colour\_Vertex}{($x,d$)} was  executed  so far with the output $c(x)=c(v)$, the condition $\delta(H)\le d-1$ holds.   By Observation \ref{obs:al} and Lemma \ref{lem:1}, it will imply that an output $c$ of \textsc{Equitable $(L,{\mathcal D}_{d-1})$-colouring} is  an $(L,{\mathcal D}_{d-1})$-colouring of $G$. Note that it is enough to show this fact for $H$ satisfying $v\in V(H)$.

By a contradiction, let $v$ be a vertex for which the output $c(v)$  does not satisfy the condition, i.e. $v$ has at least $d$ neighbors in the set of vertices for which \textsc{colour\_Vertex} was already executed with the output $c(v)$. But it is not possible because $c(v)$ was removed from $L(v)$ when the last (in the sense of the algorithm steps)  of the neighbors of $v$, say $x$, obtained the colour  $c(v)$ (\textsc{colour\_Vertex}{($x$)} removed $c(x)$ from $L(v)$ since $v\in N_G^{col}(d,x)$ in this step). \hfill$\Box$
\end{proof}

\begin{lemma}\label{com:eq}
An output colour function $c$ of {\rm \textsc{Equitable $(L,{\mathcal D}_{d-1})$-colouring}($G$)} satisfies $|C_i|\le \left\lceil |V(G)|/t\right\rceil$, where $t$ is the part of the input of \textsc{Equitable $(L,{\mathcal D}_{d-1})$-colouring}$(G)$ and $C_i=\{v\in V(G):\; c(v)=i\}$.
\end{lemma}
\begin{proof}
Recall that $\left\lceil |V(G)|/t\right\rceil=\beta+1$. We will show that there exists a partition of $V(G)$ into $\beta+1$ sets, say $W_1\cup \cdots \cup W_{\beta+1}$, such that for each $i\in \{1, \ldots ,\beta+1\}$  any two vertices $x,y$ in $W_i$ satisfy $c(x)\neq c(y)$. It will imply that the cardinality of every colour class in $c$ is at most $\beta+1$, giving the assertion.

Note that after the last, $\rho$-th execution of the {\bf for} lopp in line \ref{procedureOne-for-one} of \textsc{Equitable $(L,{\mathcal D}_{d-1})$-colouring} the list $S_{col}$ consists of the coloured vertices  of the set $V_1\setminus R$ (observe that $|V_1\setminus R|=\beta r$). The elements of $S_{col}$ are ordered in such a way that the first $x$ ones have different colours and for every $i\in \{1, \ldots ,\rho\}$ the $i$-th next $k$ elements  have different colours. Now  the \textsc{Reorder}($S_{col}$) procedure in line  \ref{reorder}  of \textsc{Equitable $(L,{\mathcal D}_{d-1})$-colouring} changes the ordering of elements of $S_{col}$ in such a way that  every $k$ consecutive elements have different colours. Since $r \in \{0,\ldots,k-1\}$, it follows that also every $r$ consecutive elements of this list have different colours. The execution of \textsc{Reorder}($S_{col}$) is always possible because of  the previous assumptions on $S_{col}$.

For $i\in \{1, \ldots ,\beta\}$ let $H_i=S_{\eta+1-((\beta-i+1)\gamma-1)}\cup S_{\eta+1-((\beta-i+1)\gamma-2)}\cup \cdots \cup S_{\eta+1-(\beta-i)\gamma}$. Thus $H_1\cup \cdots \cup H_{\beta}$ is a partition of $V_2$ into $\beta$ sets, each of the cardinality $\gamma k$. Note that the vertices of $H_i$ are coloured when \textsc{colour\_List}($S,\gamma k$)  in line \ref{colour-list-four} of  \textsc{Equitable $(L,{\mathcal D}_{d-1})$-colouring} is executed. More precisely, it is during the $i$-th execution of the {\bf while} loop in line \ref{procedureTwo-while-one} of \textsc{colour\_List}. It guarantees that the vertices of $H_i$ obtain pairwise different colours. Moreover, in line \ref{modify-list} of  \textsc{Equitable $(L,{\mathcal D}_{d-1})$-colouring} the lists of vertices of $H_i$ were modified in such a way that  the colours of $i$-th $r$ elements from the current list $S_{col}$ are removed from the list of each element in $H_i$. Hence, after the execution of \textsc{colour\_List}($S,\gamma k$)  in line \ref{colour-list-four} of  \textsc{Equitable $(L,{\mathcal D}_{d-1})$-colouring} the elements in $H_i$ obtain colours that are pairwise different and also different from all the colours of $i$-th $r$ elements from the list $S_{col}$ (recall that $S_{col}$ consists of the ordered vertices of $V_1\setminus R$). Hence, for every $i\in \{1, \ldots ,\beta\}$ the elements of $H_i$ and the $i$-th $r$ elements of $S_{col}$ have pairwise  different colours in $c$ and can constitute $W_i$. Moreover, the elements of $R$  constitute $W_{\beta+1}$. Thus $|W_{\beta+1}|=r_2$, which finishes the proof.\hfill$\Box$
\end{proof}

\begin{theorem}
For a given graph $G$ on $n$ vertices, a $t$-uniform list assignment $L$, a  $(k,d))$-partition  of $G$ the
{\rm \textsc{Equitable $(L,{\mathcal D}_{d-1})$-colouring}($G$)} algorithm returns $(L,{\mathcal D}_{d-1})$-colouring of $G$ in polynomial time.
\end{theorem}
\begin{proof}
The correctness of the algorithm follows immediately from  Lemmas \ref{com:col}, \ref{com:eq}. 
To determine the complexity of the entire \textsc{Equitable $(L,{\mathcal D}_{d-1})$-colouring} $(G)$, let us analyze first its consecutive instructions:
\begin{itemize}
    \item creating the list $S$ - the {\bf for} loop in line \ref{procedureReverse} can be done in $O(n)$ time
    \item   \textsc{colour\_List}{($L_R,r_2$)} in line \ref{colour-list-one} 
    
    This procedure recalls \textsc{colour\_Vertex}{($v,d$)} in which the most costly operation is the one depicted in line \ref{delete} of \textsc{colour\_Vertex}. Since, we colour $r_2$ initial vertices we do not need to check $N_G^{col}(d,v)$ for first $d$ vertices. If $r_2\leq d$, then the cost of \textsc{colour\_List}{($L_R,r_2$)} is $O(r_2)$, otherwise $O(r_2 \Delta^2(G))$. 
    \item  \textsc{colour\_List}{($L_X,x$)} in line \ref{colour-list-two} 
    
    Similarly as in the previous case, the complexity of this part can be bounded by $O(x\Delta^2(G))$.
    \item the {\bf for} loop in line \ref{procedureOne-for-one}
    
    This loop is executed $\eta$ times, $\rho \leq \eta \leq \lceil n/k \rceil$. The complexity of the internal \textsc{colour\_List}{($S',k$)} procedure is $O(k\Delta^2(G))$. So, we have $O(n\Delta^2(G))$ in total.
    \item \textsc{Reorder}{($S_{col}$)} in line \ref{reorder}
    
    Since $(\eta - \beta \gamma)k\leq n$, then we get $O(n)$.
    
    \item  \textsc{Modify\_colourList}{($S_{col}, \overline{S}$)}
    
    The double {\bf for} loop implies the complexity $O(\beta \gamma k)=O(n)$. 
    \item   \textsc{colour\_List}{($S,\gamma k$)}
    
    The list $S$ is of length $\beta \gamma k$, so the {\bf while} loop in line \ref{procedureTwo-while-one} of \textsc{colour\_List} is executed $\beta$ times, while the internal {\bf while} loop is executed $\gamma k$ times. Taking into account the complexity of \textsc{colour\_Vertex}, we get $O(n\Delta^2(G))$.
\end{itemize}

We get $O(n\Delta^2(G))$ as the complexity of the entire \textsc{Equitable $(L,{\mathcal D}_{d-1})$-colouring}$(G)$. \hfill$\Box$
\end{proof}

\section{Grids}\label{grids}
Given two graphs $G_1$ and $G_2$, the \emph{Cartesian product} of $G_1$ and $G_2$, $G_1 \square G_2$, is defined to be a graph whose the vertex set is
$V(G_1) \times V(G_2)$ and the edge set consists of all  edges joining vertices $(x_1,y_1)$ and $(x_2,y_2)$ when either $x_1=x_2$ and $y_1y_2 \in E(G_2)$ or $y_1=y_2$ and $x_1x_2\in E(G_1)$. Note that the Cartesian product is commutative and associative. Hence the graph $G_1 \square \cdots \square G_d$ is unambiguously defined for any $d \in \mathbb{N}$. Let $P_n$ denote a path on $n$ vertices. If each factor $G_i$ is a path on at least two vertices then $G_1 \square \cdots \square G_d$ is a $d$-\emph{dimensional grid} (cf. Fig.~\ref{p34}). Note that the $d$-dimensional grid $P_{n_1}\square \cdots \square P_{n_d}$, $d \geq 3$, may be considered as $n_1$ layers and each layer is the $(d-1)$-dimensional grid $P_{n_2}\square \cdots \square P_{n_d}$. We assume $n_1 \geq \cdots \geq n_d$.

\begin{figure}[htb]
\begin{center}
\includegraphics[scale=1]{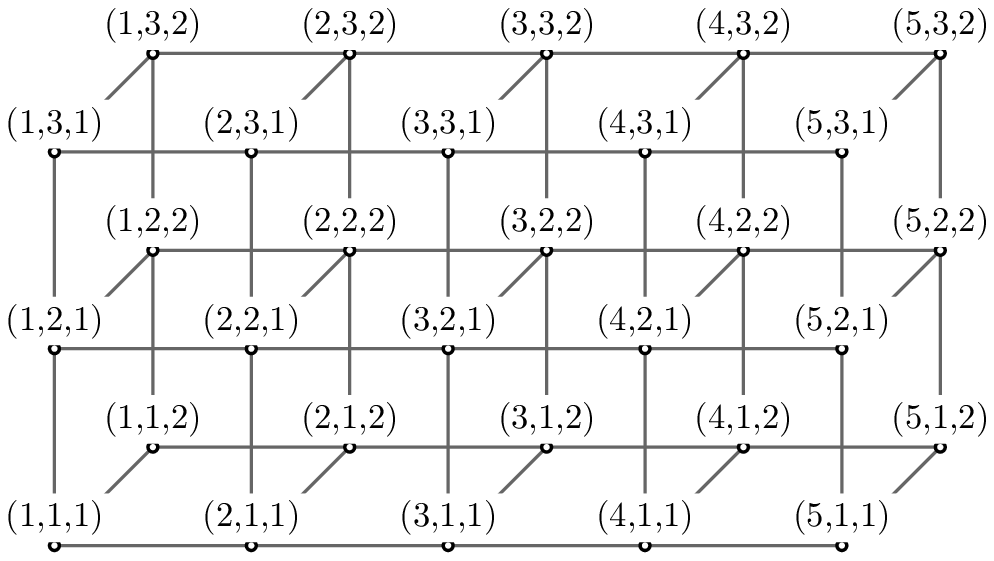}
\caption{The 3-dimensional grid $P_5 \square P_3 \square P_2$.}
\label{p34}
\end{center}
\end{figure}

Let $P_{n_1} \sqsupset \ldots \sqsupset P_{n_d}$ denote an \textit {incomplete} $d$-dimensional grid, i.e. a connected graph 
being a subgraph of $P_{n_1} \square \ldots \square P_{n_d}$ such that its some initial layers may be empty, the first  non-empty 
layer may be incomplete, while any next layer is complete (cf. Fig.~\ref{incom}). Note that every grid is particular incomplete grid.
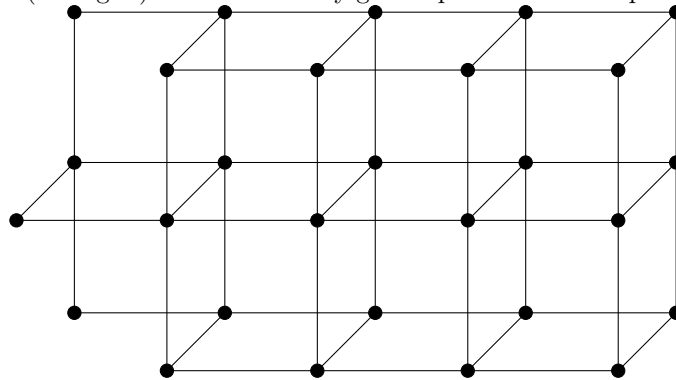
\begin{figure}[htb]
\centering
\begin{tikzpicture}

\def \dx{2};
\def \dy{2};
\def \dz{2};
\def \nbx{5};
\def \nby{3};
\def \nbz{2};

\foreach \x in {2,...,\nbx} {
    \foreach \y in {1,...,\nby} {
        \foreach \z in {1,...,\nbz} {
            \node at (\x*\dx,\y*\dy,\z*\dz)
            [circle, minimum width=5pt, fill, draw, inner sep=0pt]{};
        }
    }
}
\node at (2,2,2) [circle, minimum width=5pt, fill, draw, inner sep=0pt]{};
\node at (2,4,2) [circle, minimum width=5pt, fill, draw, inner sep=0pt]{};
\node at (2,6,2) [circle, minimum width=5pt, fill, draw, inner sep=0pt]{};
\node at (2,4,4) [circle, minimum width=5pt, fill, draw, inner sep=0pt]{};
\draw (2,2,2)--(2,4,2); 
\draw (2,4,2)--(2,6,2);
\draw (2,4,2)--(2,4,4);

\foreach \x in {2,...,\nbx} {
    \foreach \z in {1,...,\nbz}{
        \draw (\x*\dx,\dy,\z*\dz) -- ( \x*\dx,\nby*\dy,\z*\dz);
    }
}
\foreach \x in {2,...,\nbx} {
    \foreach \y in {1,...,\nby}{
        \draw (\x*\dx,\y*\dy,\dz) -- ( \x*\dx,\y*\dy,\nbz*\dz);
    }
}

\foreach \y in {1,...,\nby} {
    \foreach \z in {1,...,\nbz}{
        \draw (\dx+2,\y*\dy,\z*\dz) -- ( \nbx*\dx,\y*\dy,\z*\dz);
    }
}

\draw (2,2,2)--(4,2,2);
\draw (2,4,2)--(4,4,2);
\draw (2,6,2)--(4,6,2);
\draw (2,4,4)--(4,4,4);
\end{tikzpicture}
\caption{An incomplete grid $P_n \sqsupset P_3 \sqsupset P_2$, $n\geq 5$.}\label{incom} 
\end{figure}

In this subsection we construct a polynomial-time algorithm that for each $3$-dimensional grid  finds  its $(3,2)$-partition (\textsc{Partition3d}($G$)). Application of Theorem \ref{thm:1} implies the main result of this subsection. 
Note that using a completely different method, the first statement of Theorem \ref{main3} has already been proven in \cite{DrFrDy18}.

\begin{theorem}\label{main3}
Let $t\geq 3$ be an integer. Every $3$-dimensional grid is equitably  $(t,{\mathcal D}_1)$-choosable. Moreover, there is a polynomial-time algorithm that for every $t$-uniform list assignment $L$ of the $3$-dimensional grid $G$ returns an equitable $(L,{\mathcal D}_{1})$-colouring of $G$.
\end{theorem}

\begin{algorithm}
\renewcommand{\algorithmcfname}{Procedure}
\caption{\textsc{Corner}($G$)}
\Input {Incomplete non-empty $d$-dimensional grid $G=P_{n_1} \sqsupset \ldots \sqsupset P_{n_d}$, $d \geq 2$.}
\Output {Vertex $y=(a_1,\ldots,a_d) \in V(G)$ such that $\deg_G(y)\leq d$. }
initialization\;
let $a_1$ be the number of the incomplete layer of $G$;\\
\For{$i=2$ \textbf{\emph{to}} $d-1$}{
$a_i=\min\{x_i: \exists_{x_{i+1},\ldots, x_d}(a_1,\ldots, a_{i-1}, x_i, \ldots,x_d)\in V(G)\}.$}
$a_d:=\min\{x_d: (a_1,\ldots, a_{d-1}, x_d)\in V(G)\}.$\\
\textbf{return} $(a_1,\ldots,a_d)$;
\end{algorithm}

\begin{algorithm}[htb]
\caption{\textsc{Partition3d}($G$)}\label{alg3d}
\Input{$3$-dimensional grid $G=P_{n_1} \square P_{n_2} \square P_{n_3}$.}
\Output {A $(3,2)$-partition $S_1 \cup \cdots \cup S_{\alpha+1}$ of $G$.}
initialization;\\
$\alpha:=\lceil \frac{n_1n_2n_3}{3} \rceil -1$;\\
\If {$\alpha \geq 1$}{\label{ifalpha}
\For{$j:=\alpha +1$ \textbf{\emph{downto}} $2$}{\label{firstfor}
$y_1^j=(a_1,a_2,a_3):=$\textsf{Corner}($G$);\label{firstcorner}\\
\If {$\deg(y_1^j)=1$}{\label{deg1}
$y_2^j:=$\textsf{Corner}($G - y_1^j$);\label{secondcorner}\\
\eIf {$y_1^j$ is the only vertex on $a_1$ layer}{
let $y_3^j$ be any vertex on layer $a_1+1$ such that $y_3^j \neq y_2^j$}
{
let $y_3^j$ be any vertex on layer $a_1$ such that $y_3^j \neq y_1^j$ and $y_3^j \neq y_2^j$, if exists, otherwise $y_3^j$ is any vertex on layer $a_1+1$}}
\If {$\deg(y_1^j)=2$}{\label{deg2}
let $y_2^j$ be the neighbour of $y_1^j$ lying on the same layer as $y_1^j$; \\
let $y_3^j$ be any vertex on layer $a_1$, if exists,
otherwise, $y_3^j$ is any vertex on layer $a_1+1$}
\If {$\deg(y_1^j)=3$}{\label{deg3}
$y_2^j:=(a_1,a_2,a_3+1)$; 
$y_3^j:=(a_1,a_2+1,a_3)$;
}
$S_j:=\{y_1^j,y_2^j,y_3^j\}$;
$G:=G-S_j$;\\
}
}
$S_1:=V(G)$;\label{rest}\\
\end{algorithm}

\begin{theorem}
For a given $3$-dimensional grid $G$ the 
{\rm {\textsc{Partition3d}($G$)}} algorithm returns
a $(3,2)$-partition of $G$ in polynomial-time.\label{3grid}
\end{theorem}
\begin{proof} First, observe that thanks to the condition of the \textbf{if} instruction in line \ref{ifalpha}, the \textbf{for} loop in line \ref{firstfor} is correctly defined and its instruction are performed every time for a graph with at least three vertices. Thus, the  \textsc{Corner} procedure in both line \ref{firstcorner} and line \ref{secondcorner} is called for a non-empty graph, so the returned vertices $y_1^j$ and $y_2^j$ are correctly defined. There is no doubt that the remaining instruction are also correctly defined, and, in consequence, we get sets $S_1,\ldots, S_{\alpha+1}$ fulfilling conditions: $|S_1|\leq 3$ and $|S_j| = 3$, $j \in \{2,\ldots,\alpha+1\}$. 

Hence, all we need is to prove that the partition $S_1 \cup \cdots \cup S_{\alpha+1}$ fulfills also the condition (\ref{part-cond}) in the definition of the $(k,s)$-partition of $G$. 
Indeed, let us consider set $S_j=\{y_1^j, y_2^j, y_3^j\}$. If the condition from  line \ref{deg1} of \textsc{Partition}{3d$(G)$} is true, then $\deg_{G-S_j}(y_1^j)=1$, $\deg_{G-S_j}(y_2^j)\leq 3$, while $\deg_{G- S_j}(y_3^j)\leq 5$. The same inequalities hold whenever the condition in line \ref{deg2} holds. If the condition in  line \ref{deg3} is true, then $\deg_{G- S_j}(y_1^j)\leq 1$, $\deg_{G-S_j}(y_2^j)\leq 3$, while $\deg_{G-S_j}(y_3^j)\leq 4$.

It is easy to see that the complexity of  \textsc{Partition}{3d$(G)$} is linear due to the number of vertices of $G$. \hfill$\Box$
\end{proof}

As a consequence of the above theorem and Theorem \ref{thm:1} we get the statement of Theorem \ref{main3}.

\section{Concluding remarks}
In Subsection \ref{subal} we have proposed the polynomial-time algorithm that finds an equitable $(L,{\mathcal D}_{d-1})$-colouring of a given graph $G$  assuming that we know a $(k,d)$-partition of $G$ ($L$ is a $t$-uniform list assignment for $G$, $t\geq k$). In this context the following  open question seems to be interesting:  What is the complexity of  recognition of graphs having a $(k,d)$-partition? 

\textbf{Acknowledgment.} The authors thank their colleague Janusz Dybizbański for making several useful suggestions improving the presentation.

%
%
%
%

\end{document}